\documentclass{amsart}

\usepackage{amssymb}

\usepackage{amsthm}

\usepackage[brazil,spanish,english]{babel}

\usepackage{dsfont}

\usepackage{graphicx}

\usepackage{hyperref}

\usepackage[latin1]{inputenc}

\usepackage{epstopdf}

\usepackage{srcltx}

\usepackage[normalem]{ulem}

\usepackage{vruler}

\usepackage{color}

\usepackage[hmargin=30mm,top=30mm,bottom=35mm]{geometry}

\theoremstyle{plain}

\newtheorem{theorem}{Theorem}

\newtheorem{lemma}[theorem]{Lemma}

\newtheorem{sub-lemma}[theorem]{Sub-lemma}

\newtheorem{proposition}[theorem]{Proposition}

\newtheorem*{proposition*}{Proposition}

\newtheorem*{theorem*}{Theorem}

\theoremstyle{definition}

\newtheorem*{definition*}{Definition}

\theoremstyle{remark}

\newtheorem*{claim*}{Claim}

\newtheorem*{remark*}{Remark}

\newcommand{\T}{\mathbb{T}}

\newcommand{\R}{\mathbb{R}}

\newcommand{\Z}{\mathbb{Z}}

\newcommand{\Q}{\mathbb{Q}}

\renewcommand{\geq}{\geqslant}

\renewcommand{\leq}{\leqslant}

\begin{document}

{\Large \it \large This paper is dedicated to the memory of Lauro Antonio Zanata}


\vskip2.0truecm

\title{Rational mode locking for homeomorphisms of the $2$-torus}

\author{Salvador Addas-Zanata}

\address{Instituto de Matem\'atica e Estat\'\i stica, Universidade de S\~ao Paulo, Rua do Mat\~ao 1010, Cidade Universit\'aria, 05508-090 S\~ao Paulo, SP, Brazil -- Salvador is partially supported by CNPq grant 306348/2015-2}

\email{sazanata@ime.usp.br}

\author{Patrice Le Calvez}

\address{Sorbonne Universit\'es, UPMC Univ Paris 06,\\ Institut de Math\'ematiques de Jussieu-Paris Rive Gauche, UMR 7586, CNRS\\ Univ Paris Diderot, Sorbonne Paris Cit\'e\\F-75005, Paris, France -- Patrice is partially supported by CAPES, Ciencia Sem Fronteiras, 160/2012}

\email{patrice.le-calvez@imj-prg.fr}

 \date{}

\begin{abstract}

In this paper we consider homeomorphisms of the torus $\R^2/\Z^2$, homotopic to the identity, and their rotation sets. Let $f$ be such a homeomorphism, $\widetilde{f}:\R^2\to\R^2$ be a fixed lift and $\rho (\widetilde{f})\subset\R^2$ be its rotation set, which we assume to have interior. We also assume that the frontier of $\rho (\widetilde{f})$ contains a rational vector $\rho\in\Q^2$ and we want to understand how stable this situation is. To be more precise, we want to know if it is possible to find two different homeomorphisms $f_1$ and $f_2$, arbitrarily small $C^0$-perturbations of $f$, in a way that $\rho$ does not belong to the rotation set of $\widetilde f_1$ but belongs to the interior of the rotation set of $\widetilde f_2,$ where $\widetilde f_1$ and $\widetilde f_2$ are the lifts of $f_1$ and $f_2$ that are close to $\widetilde f$. We give two examples where this happens, supposing $\rho=(0,0)$. The first one is a smooth diffeomorphism with a unique fixed point lifted to a fixed point of $\widetilde f$. The second one is an area preserving version of the first one, but in this conservative setting we only obtain a $C^0$ example. We also present two theorems in the opposite direction. The first one says that if $f$ is area preserving and analytic, we cannot find $f_1$ and $f_2$ as above. The second result, more technical, implies that the same statement holds if $f$ belongs to a generic one parameter family $(f_t)_{t\in[0,1]}$ of $C^2$-diffeomorphisms of $\T^2$ (in the sense of Brunovsky). In particular, lifting our family to a family $(\widetilde f_t)_{t\in[0,1]}$ of plane diffeomorphisms, one deduces that if there exists a rational vector $\rho$ and a parameter $t_*\in(0,1)$ such that $\rho (\widetilde{f}_{{t_*}})$ has non empty interior, and $\rho\not\in \rho (\widetilde{f}_t)$ for $t<t_*$ close to $t_*$, then $\rho\not\in \mathrm{int}(\rho (\widetilde{f}_{t}))$ for all $t>t_*$ close to $t_*$. This kind of result reveals some sort of local stability of the rotation set near rational vectors of its boundary.

 \end{abstract}

\maketitle










































\def\eqalign#1{\null\,\vcenter{\openup\jot

\ialign{\strut\hfil$\displaystyle{##}$&

$\displaystyle{{}##}$\hfil \crcr #1\crcr }}\,}

\def\eqalignno#1{\displ@y \tabskip=\@centering

\halign to\displaywidth{\hfil$\@lign\displaystyle{##}$

\tabskip=0pt &$\@lign\displaystyle{{}##}$

\hfil\tabskip=\@centering

$\llap{$\@lign##$}\tabskip=Opt\crcr #1\crcr}}



\section{Introduction and main results}

The main motivation for this paper is to study how the rotation set of a homeomorphism of the two dimensional torus $\T^2$, homotopic to the identity, changes as the homeomorphism changes. For instance, suppose we consider a one parameter continuous family $f_t:{\T^2\rightarrow \T^2}$ of such maps, a continuous family of lifts $\widetilde f_t:{\R^2\rightarrow \R^2}$ and want to study how the parameterized family of rotation sets $t\longmapsto \rho (\widetilde{f}_t)$ varies. The rotation set is a non empty compact convex subset of the plane (see definition below), which varies continuously with the homeomorphism, at least in the special situation when it has interior (see \cite{mz2}). In particular, we are interested in the following problem: Suppose $f:{\T^2\to\T^2}$ is a homeomorphism homotopic to the identity and its rotation set for a given lift $\widetilde f$, which is supposed to have interior, has a point $\rho $ in its boundary with both coordinates rational. Is it possible to find two different arbitrarily small $C^0$-perturbations of $f,$ denoted $f_1$ and $f_2$, in a way that $\rho $ does not belong to the rotation set of the lift of $f_1$ close to $\widetilde f$ but is contained in the interior of the rotation set of the lift of  $f_2$ close to $\widetilde f$? In other words we are asking if the rational mode locking found by de Carvalho, Boyland and Hall \cite{andre} in their particular family of homeomorphisms is, in a certain sense, a general phenomena or not. Our main theorems and examples will show that the answer to this question depends on the hypotheses we have. In general, we can find such maps $f_1$ and $f_2$ as described above, but if we assume certain hypotheses on $f,$ then this sort of ``local mode locking'' happens.

Even in the much simpler context of orientation preserving circle homeomorphisms, the only maps with rational rotation number which can be perturbed in an arbitrarily $C^0$-small way in order to decrease or increase their rotation numbers (the analogous one dimensional version of our condition) are the ones conjugate to rational rotations. In the context of degree one circle endomorphisms, if $f:\T^1\rightarrow \T^1$ is such an endomorphism, $\widetilde{f}:\R\rightarrow \R$ is a fixed lift and $\rho (\widetilde{f})$ is its rotation set (which in this case is a compact interval), if we assume that the rotation set is not reduced to a point and has a rational end $\rho$, then it is not possible to find $C^0$ neighbors of $\widetilde f$ in the space of lifts, one whose rotation set does not contain $\rho$ and the other with $\rho$ in the interior of its rotation set (see Theorem \ref{endomor} below).

In order to make things precise and to present our main results, a few definitions are necessary:

\smallskip

\noindent -\enskip We denote ${\T^2}=\R^2/\Z^2$  the flat torus and $\pi:\R^2\to\T^2$ the universal covering projection.

\smallskip

\noindent -\enskip We denote $\mathrm{Diff}_0^r({\T^2})$ the space of $C^r$ (for $r=0,1,2,...,\infty ,\omega $) diffeomorphisms (homeomorphisms if $r=0$) of the torus homotopic to the identity and  $\widetilde{\mathrm{Diff}}_0^r({\T^2})$ the space of lifts of elements of $\mathrm{Diff}_0^r({\T^2})$ to the plane.

\smallskip

\noindent -\enskip We write $p_{1}: (x,y)\mapsto x$ and $p_{2}:(x,y)\mapsto y$ for the standard projections defined on the plane.

\smallskip

\noindent -\enskip Given $f\in \mathrm{Diff}_0^0({\T^2})$ and a lift $\widetilde f\in  \widetilde{\mathrm{Diff}}_0^0({\T^2})$ the {\it rotation set} $\rho (\widetilde{f})$ of $\widetilde{f}$  can be defined following Misiurewicz and Ziemian \cite{misiu} as: 
$$\rho (\widetilde{f})=\bigcap_{i\geq1} \overline{\bigcup_{n\geq i}\left\{ \frac{\widetilde{f}^n(\widetilde{z})-\widetilde{z}}{n},\widetilde{z}\in \R^2\right\} }$$

This set is a compact convex subset of ${\R^2}$ (see \cite {misiu}), and it was proved in \cite{franksrat} and \cite{misiu} that all points in its interior are realized by compact $f$-invariant subsets of ${\T^2}$, which can be chosen as periodic orbits in the rational case. By saying that some vector $\rho \in \rho (\widetilde{f})$ is realized by a compact $f$-invariant set, we mean that there exists a compact $f$-invariant subset $K\subset {\T^2}$ such that for all $z\in K$ and any $\widetilde{z}\in \pi^{-1}(\{z\})$, one has 
$$\lim_{n\rightarrow \infty }\frac{\widetilde{f}^n(\widetilde{z})-\widetilde{z}}n=\rho .$$
Moreover, the above limit, whenever it exists, is called the {\it rotation vector} of $z$ and denoted $\rho (z)$.

As the rotation set is a compact convex subset of the plane, there are three possibilities for its shape: it is a point,  a linear segment or it has interior. In this paper we only consider the situation when the rotation set has interior. The first main result is:

\begin{theorem}

\label{main1} Let $\widetilde f\in \widetilde{\mathrm{Diff}}_0^{\omega}({\T^2})$ be an area preserving diffeomorphism such that the interior of $\rho (\widetilde{f})$ is non empty and its frontier contains a rational vector $\rho\in\Q^2$. Then, one of the following situations occurs:

\smallskip

\noindent-\enskip there exists a neighborhood $\mathcal U$ of $\widetilde f$ in $\widetilde{\mathrm{Diff}}_0^0({\T^2})$ such that $\rho\in \rho (\widetilde{f}')$ for every ${\widetilde f'}\in\mathcal U$;

\smallskip

\noindent-\enskip there exists a neighborhood $\mathcal U$ of $\widetilde f$ in $\widetilde{\mathrm{Diff}}_0^0({\T^2})$ such that $\rho\not\in \mathrm{int}(\rho (\widetilde{f}'))$ for every ${\widetilde f'}\in\mathcal U$.

\end{theorem}

The next result permits to understand the behaviour of generic one parameters families.  Fix $\rho=(p/q,r/q)\in\Q^2$, where $q\geq 1$ and $\mathrm{g.c.d}(p,q,r)=1$. Suppose that $\widetilde f\in\widetilde{\mathrm{Diff}}_0^0({\T^2})$ can be approximated arbitrarily close by $\widetilde f'\in \widetilde{\mathrm{Diff}}_0^0({\T^2})$ such that  $\rho\in \mathrm{int}(\rho (\widetilde{f}'))$. One deduces that the fixed point set of $\widetilde f^{q}-(p,r)$ is not empty. Suppose moreover that this set is discrete (which means that it projects onto  a finite set of $\T^2$) and that $\widetilde f$ may be approximated arbitrarily close by $\widetilde f''\in\widetilde{\mathrm{Diff}}_0^0({\T^2})$ such that  $\rho\not\in \rho (\widetilde{f}'')$. In that case, the Lefschetz index of every fixed point $\widetilde z$ of $\widetilde f^{q}-(p,r)$ is equal to $0$. In particular,  if $\widetilde f\in\widetilde{\mathrm{Diff}}_0^1({\T^2})$, one knows that $1$ is an eigenvalue of $D\widetilde f^q(\widetilde z)$.  

Let us say that a diffeomorphism $\widetilde f\in\mathrm{Diff}_0^1({\T^2})$ is  {\it not highly degenerate} at $\rho$, if the fixed point set of $\widetilde f^{q}-(p,r)$ is discrete and if for every point $\widetilde z$ in this set,  $ Df^q(\widetilde z)$ has at least one eigenvalue different from $1$.

\begin{theorem}

\label{main2} Let $\widetilde f\in \widetilde{\mathrm{Diff}}_0^2({\T^2})$ be such that the interior of $\rho (\widetilde{f})$ is non empty and its frontier contains a rational vector $\rho\in\Q^2$. Suppose moreover that $\widetilde f$ is  {\it not highly degenerate} at $\rho$. Then, one of the following situations occurs:

\smallskip

\noindent-\enskip there exists a neighborhood $\mathcal U$ of $\widetilde f$ in $\widetilde{\mathrm{Diff}}_0^0({\T^2})$ such that $\rho\in \rho (\widetilde{f}')$ for every ${\widetilde f'}\in\mathcal U$;

\smallskip

\noindent-\enskip there exists a neighborhood $\mathcal U$ of $\widetilde f$ in $\widetilde{\mathrm{Diff}}_0^0({\T^2})$ such that $\rho\not\in \mathrm{int}(\rho (\widetilde{f}'))$ for every ${\widetilde f'}\in\mathcal U$.

\end{theorem}

This result implies the statement about generic families explained in the abstract. More precisely, for $C^r$-generic one parameter families of diffeomorphisms, $r\geq 1$,  Brunovski has shown, see \cite{bru1}, that the only bifurcations that create or destroy periodic points are saddle-nodes and period doubling. Theorem \ref{main2} hypotheses imply that the creation of the periodic orbits with rotation vector $\rho=(p/q,r/q)$ had to be through a saddle-node type of bifurcation because the map $f$ has neighbors without $q$-periodic points with rotation vector equal to $\rho$.

The next two results indicate the requirment of the hypotheses in the previous theorems.

\begin{theorem}

\label{main3} There exists $\widetilde f\in  \widetilde{\mathrm{Diff}}_0^{\infty}({\T^2})$ such

\smallskip

\noindent-\enskip for every neighborhood $\mathcal U$ of $\widetilde f$ in $\widetilde{\mathrm{Diff}}_0^{\infty}({\T^2})$ there exists ${\widetilde f'}\in\mathcal U$ such that $(0,0)\not\in \rho (\widetilde{f}')$;

\smallskip

\noindent-\enskip for every neighborhood $\mathcal U$ of $\widetilde f$ in $\widetilde{ \mathrm{Diff}}_0^0({\T^2})$ there exists ${\widetilde f'}\in\mathcal U$ such that $(0,0)\in \mathrm{int}(\rho (\widetilde{f}'))$;

\smallskip

\noindent-\enskip $\widetilde f$ has a unique fixed point, up to translation by a vector of $\Z^2$.

\end{theorem}

The example in Theorem \ref{main3} is not area preserving. We can construct an area preserving example, but it will not be differentiable

\begin{theorem}

\label{main4} There exists $\widetilde f\in\widetilde{\mathrm{Diff}}_0^0({\T^2})$ such

\smallskip

\noindent-\enskip for every neighborhood $\mathcal U$ of $\widetilde f$ in $\widetilde{\mathrm{Diff}}_0^0({\T^2})$ there exists ${\widetilde f'}\in\mathcal U$ such that $ (0,0)\not\in \rho (\widetilde{f}')$;

\smallskip

\noindent-\enskip for every neighborhood $\mathcal U$ of $\widetilde f$ in $\widetilde{\mathrm{Diff}}_0^0({\T^2})$ there exists ${\widetilde f'}\in\mathcal U$ such that $(0,0)\in \mathrm{int}(\rho (\widetilde{f}'))$;

\smallskip

\noindent-\enskip $\widetilde f$ has a unique fixed point, up to translation by a vector of $\Z^2$;

\smallskip

\noindent-\enskip $\widetilde f$ is area preserving.

\end{theorem}

To conclude, we present a simple theorem about endomorphisms of the circle, which shows that in this situation, the conclusion of Theorem \ref{main1} holds in full generality. The proof we present is due to Andr\'es Koropecki. Write $\pi:\R\to\T$ for the universal covering projection of ${\T}=\R/\Z$. Denote $\mathrm{End}_0({\T^1})$ the space of continuous maps $f:\T\to\T$ homotopic to the identity and $\widetilde{\mathrm{End}}_0({\T})$ the space of lifts of elements of $\mathrm{End}_0({\T})$ to the line. The {\it rotation set} $\rho (\widetilde{f})$ of $\widetilde{f}\in \widetilde{\mathrm{End}}_0{\T})$ is a real segment defined as: 
$$\rho (\widetilde{f})=\bigcap_{i\geq1} \overline{\bigcup_{n\geq i}\left\{ \frac{\widetilde{f}^n(\widetilde{z})-\widetilde{z}}{n},\widetilde{z}\in \R\right\} }.$$ 

Let $\widetilde f\in\widetilde{\mathrm{End}}_0({\T})$ be such that $\rho (\widetilde{f})$ is not reduced to a point and its frontier contains a rational number $\rho$. We will prove that one of the following situations occurs:

\smallskip

\noindent-\enskip there exists a neighborhood $\mathcal U$ of $\widetilde f$ in $\widetilde{\mathrm{End}}_0({\T})$ such that $\rho\in \rho (\widetilde{f}')$ for every ${\widetilde f'}\in\mathcal U$;

\smallskip

\noindent-\enskip there exists a neighborhood $\mathcal U$ of $\widetilde f$ in $\widetilde{\mathrm{End}}_0({\T})$ such that $ \rho\not\in \mathrm{int}(\rho (\widetilde{f}'))$ for every ${\widetilde f'}\in\mathcal U$.

In fact we have the stronger following result:

\begin{theorem}

\label{endomor} Let $\rho=p/q$ be a rational number and $\widetilde f\in\widetilde{\mathrm{End}}_0({\T})$ such that $\widetilde{f}^{q}-p\not=\mathrm{Id}$. Then, one of the following situations occurs:

\smallskip

\noindent-\enskip there exists a neighborhood $\mathcal U$ of $\widetilde f$ in $\widetilde{\mathrm{End}}_0({\T})$ such that $\rho\in \rho (\widetilde{f}')$ for every ${\widetilde f'}\in\mathcal U$;

\smallskip

\noindent-\enskip there exists a neighborhood $\mathcal U$ of $\widetilde f$ in $\widetilde{\mathrm{End}}_0({\T})$ such that $ \rho\not\in \mathrm{int}(\rho (\widetilde{f}'))$ for every ${\widetilde f'}\in\mathcal U$.

\end{theorem}

\begin{proof} If the map $\widetilde{f}^{q}-p-\mathrm{Id}$ takes a positive and a negative value, it will be the same for $\widetilde f'^{q}-p-\mathrm{Id}$ if $\widetilde f'\in \widetilde{\mathrm{End}}_0({\T})$ is close to $f$. This implies that $\widetilde f'^{q}-p$ has at least one fixed point, so the first assertion is true. If $\widetilde{f}^{q}-p-\mathrm{Id}$ does not vanish, it will be the same for $\widetilde f'^{q}-p-\mathrm{Id}$ if $\widetilde f'\in \widetilde{\mathrm{End}}_0({\T})$ is close to $f$. This implies that $\rho$ does not belong to $\rho (\widetilde f ')$, so the second assertion is true. As we suppose that $\widetilde{f}^{q}-p-\mathrm{Id}$ is not identically zero, it remains to study the case where it vanishes but has constant sign. Suppose that there exists $x_0$ such that $\widetilde g(x_0)<x_0$, where $\widetilde g=\widetilde f^q-p$. By hypothesis $\widetilde g(x)\leq x$ for every $x\in\R$ and so $\widetilde g([x_0-1,x_0])\subset (-\infty, x_0)$. We deduce that there exists a neighborhood $\mathcal U$ of $\widetilde f$ in $\widetilde{\mathrm{End}}_0({\T})$ such that $\widetilde g'([x_0-1,x_0])\subset (-\infty, x_0)$, where $\widetilde g'=\widetilde f'^q-p$. Consequently, one gets $g'((-\infty,x_0])\subset (-\infty, x_0)$ and more generally $\widetilde g'^{n}((-\infty,x_0])\subset (-\infty, x_0)$  for every $n\geq 1$. The fact that 
$$x\leq x_0\text{ and } n\geq 1\Rightarrow g^n(x)\leq x_0$$
implies that the rotation set of $f'$ in included in $(-\infty,p/q]$. Similarly, if $\widetilde{f}^{q}-p-\mathrm{Id}$ vanishes and is non negative, there exists a neighborhood $\mathcal U$ of $\widetilde f$ in $\widetilde{\mathrm{End}}_0({\T})$ such that $\rho(\widetilde f')\subset [p/q,+\infty)$.\end{proof}

This paper is organized as follows. In the next section we present a brief summary on the local dynamics near fixed points of area preserving analytic plane diffeomorphisms. In the third section we prove a fundamental result necessary to get Theorems 1 and 2. The precise proofs of Theorems 1,2, 3 and 4 will be given in the fourth section.

\section{local study of analytic and area-preserving planar diffeomorphisms}

The dynamics near isolated singularities of analytic vector fields in the plane is very well understood, at least when the topological index of the singularity is not $1$ (see for instance Dumortier \cite{singvet}). It can be proved that, if the singularity is neither a focus or a center (which have topological index equal $1$), then the dynamics near it can be obtained from a finite number of sectors, glued in an adequate way. Topologically, these sectors can be classified in four types:\ elliptic, hyperbolic, expanding and attracting. Dumortier, Rodrigues and Roussarie studied this problem for planar diffeomorphisms near fixed points in \cite{Dur}. Let us recall a fundamental notion in their study: a smooth planar vector field $X$, vanishing at the origin has {\it Lojasiewicz type of order $k\geq 1$} if  there exist $C>0$ and $\delta >0$ such that $\Vert x\Vert <\delta\Rightarrow \Vert X(x)\Vert \geq C\Vert x\Vert ^k$. It can be easily seen that this notion depends only on the $k$-jet $J^k_X$ at $(0,0)$. In particular $X$ has Lojasiewicz type of order $k\geq 1$ if and only if it is the case for $J^k_X$ and one can extend this notion to any formal vector field by looking at the truncated series of order $k$. If $X$ is a real analytic planar vector field vanishing at the origin, it has Lojasiewicz type, which means that there exists $k\geq 1$ such that $X$ has Lojasiewicz type of order $k\geq 1$, (see Bierstone-Milman \cite{lojaana}).

The situation we want to understand in this sub-section is the following: Is there a topological picture of the dynamics near an index zero isolated fixed point of an analytic area-preserving planar diffeomorphism? It turns out that the area-preservation together with the zero index hypothesis imply that the eigenvalues of the derivative of the diffeomorphism at the fixed point are both equal to $1$. We will suppose that the fixed point is the origin. The area-preservation implies that the infinite jet $J^{\infty}_f$ of $f$ at $(0,0)$ is the time one mapping of a unique formal vector field $\widetilde X$ (defined by a formal series), see  \cite{takens} and Moser \cite{moser}.  Let us fix a smooth vector field $X$ such that $J^{\infty}_X=\widetilde X$ and consider the time one map $g$ of the flow defined by $X$. Let us prove by contradiction that $\widetilde X$ has Lojasiewicz type. If not, by Proposition 2 of Llibre-Saghin \cite{LS}, one can choose $X$ such that  $0$ is a non isolated zero of $X$ and consequently a non isolated fixed point of $g$. One deduces that $J^{\infty}_{g-\mathrm{Id}}$ has not Lojasiewitz type. On the other hand, $J^{\infty}_{f-\mathrm{Id}}$ has Lojasiewitz type because $f$ is analytic. So, we have a contradiction because $f$ and $g$ have the same infinite jet at $(0,0)$.

Let us explain now why $X$ and $f-\mathrm{Id}$ have the same index at $0$. The fact that $f-\mathrm{Id}$ and $g-\mathrm{Id}$ have the same index (which is equal to zero) at the origin is given by Proposition 1 of \{\cite{LS}. It is a consequence of the fact that these vector fields have the same infinite jet at $0$ and that this jet has Lojasiewicz type. Write $g_t$, $t\in(0,1]$,  for the time $t$ map of the flow induced by $X$. There exists $\delta>0$ such that the vector field $X$ and the vector fields $g_t-\mathrm{Id}$, $t\in(0,1]$ have no zero satisfying  $0<\Vert x\Vert\leq \delta$. Otherwise, $g=g_1$ has an invariant curve in any neighborhood of $0$, which implies that the index of $g-\mathrm{Id}$ is $1$, in contradiction with the hypothesis. By computing the indices on the circle of equation $\Vert x\Vert =\delta$, one deduces that the indices of the $g_t-\mathrm{Id}$ are all the same. By letting $t$ tend to $0$, one deduces that this common index is the index of $X$.

The fact that the index of $X$ is not $1$ implies, by \cite {singvet}, that $X$ has at least one characteristic orbit at $0$, which means an integral curve $\gamma$ of $X$ or $-X$ defined on $[0,\infty)$ such that:
$$\gamma(t)\not=0, \enskip \lim_{t\to+\infty}\gamma(t)=0,\enskip \lim_{t\to+\infty} \gamma(t)/\Vert \gamma(t)\Vert \text{ exists}.$$ The existence of a characteristic orbit only depends on a finite jet of $\widetilde{X}$ and is independent of the choice of $X$. The fact that $\widetilde X$ has Lojasiewicz type and has a characteristic orbit permits us to apply \cite{Dur}, Theorem D:   there exists a vector field $X'$ such that $J_{X'}^{\infty}=\widetilde X$ and such that $f$ is weakly-$C^0$-conjugated to the time-$1$ map of the flow induced by $X'$. 

 The above theorem implies, see \cite{Dur} pages 39-40, that the dynamics of $f$ in a neighborhood of the origin is obtained by gluing a finite number of sectors, which can be attracting, expanding, elliptic or hyperbolic and moreover, in our situation, as we are supposing that $f$ preserves the area, there can not be elliptic, expanding and attracting sectors. As the topological index of the fixed point is $0,$ there must be exactly two invariant hyperbolic sectors and the dynamics is topologically as in Figure 1.

\begin{figure}[ht!]
\hfill
\includegraphics [height=48mm]{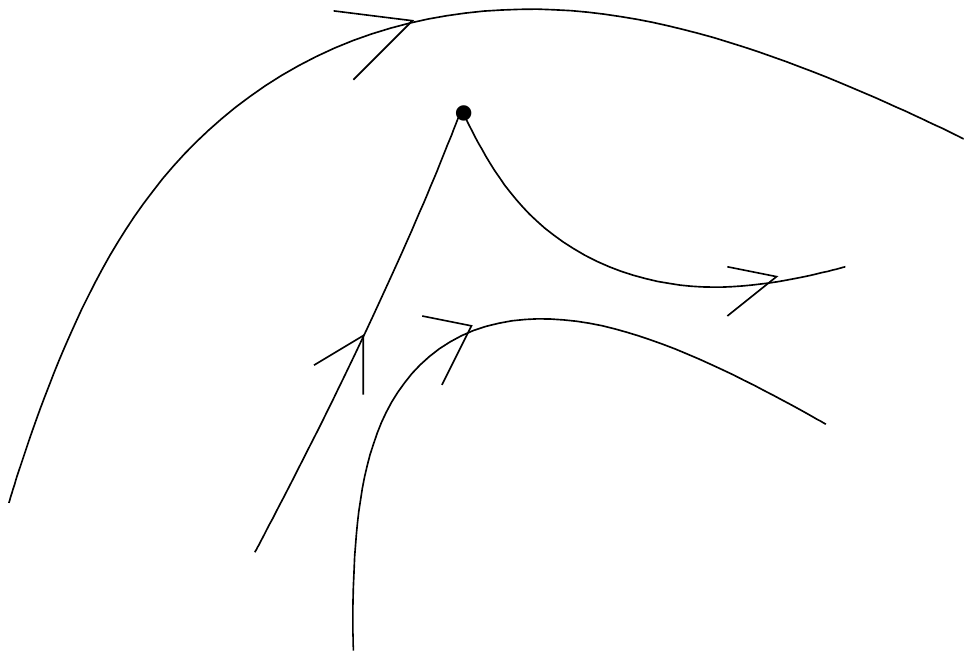}
\hfill{}
\caption{\small }
\end{figure}

\section{Fundamental proposition}

We introduce in this section an important notion for our problem. Let $f$ be a homeomorphism of a surface $M$. A fixed point $z_0$ of $f$ will be called {\it trivializable} if there exists a continuous chart $h:U\subset M\to V\subset \R^2$  at $z_0$ such for every 
$z\in U\cap f^{-1}(U)$, one has $p_1\circ h\circ f(z)>p_1\circ h(z)$ if $z\not=z_0$.

Let us give two categories of trivializable fixed points.

\begin{proposition}

\label{saddle-node} 

Let $f$ be a $C^2$-diffeomorphism of a surface $M$ and $z_0$ a fixed point that is a saddle-node, meaning that one of the eigenvalue of $Df(z_0)$ is $1$ and the other one is not. Then $z_0$ is trivializable.

\end{proposition}

\begin{proof}

By considering $f^{-1}$ if necessary, we can assume that $\det (Df(z_0))<1$. This proposition follows from the following lemma (for example, see Carr \cite{Carr}, Theorem 1, page 16 and Lemma 1, page 20, where the results are proved for vector fields, analogous proofs holding for maps, as stated in pages 33-35):

\begin{lemma}

\label{carrr}

 Assume $f:\R^2\rightarrow\R^2$ is a $C^2$-diffeomorphism which fixes the origin and $Df(0,0)$ has $1$ and $\lambda \in(0,1)$ as eigenvalues. If we write $f$ in coordinates such that $f(x,y)=(x+u(x,y),\lambda y+v(x,y))$, where the functions $u$, $v$ and their first derivatives vanish at the origin, then there exists a $C^2$-function $h$ defined for  $\vert x\vert$ sufficiently small such that $h(0)=h^{\prime }(0)=0$, whose graph is invariant under iterates of $f$, a neighborhood $U$ of $(0,0)$ and $C>0$ such that for any segment of orbit $(x_k,y_k)_{0\leq k\leq n}$ included in $U$, one has $\left| y_n-h(x_n)\right| \leq C.\lambda ^n.\left| y_0-h(x_0)\right| $.

\end{lemma}

As we are assuming that the origin is an isolated fixed point, which is a saddle-node, without loss of generality we can suppose that points in the center manifold with negative $x$ coordinate converge to the origin under positive iterates of $f$ and points in the center manifold, with positive $x$ coordinate converge to the origin under negative iterates of $f.$

So, from Lemma \ref{carrr} the dynamics in a neighborhood of the origin can be obtained by gluing exactly three sectors, two adjacent hyperbolic sectors and one attracting one. Thus, under a $C^0$-coordinate change, it is easy to see that the fixed point is trivializable: the vertical foliation in this system of coordinates is topologically transverse to the natural foliation by locally invariant leaves defined by the sadlle-node (see Figure 2).

\begin{figure}[ht!]
\hfill
\includegraphics [height=48mm]{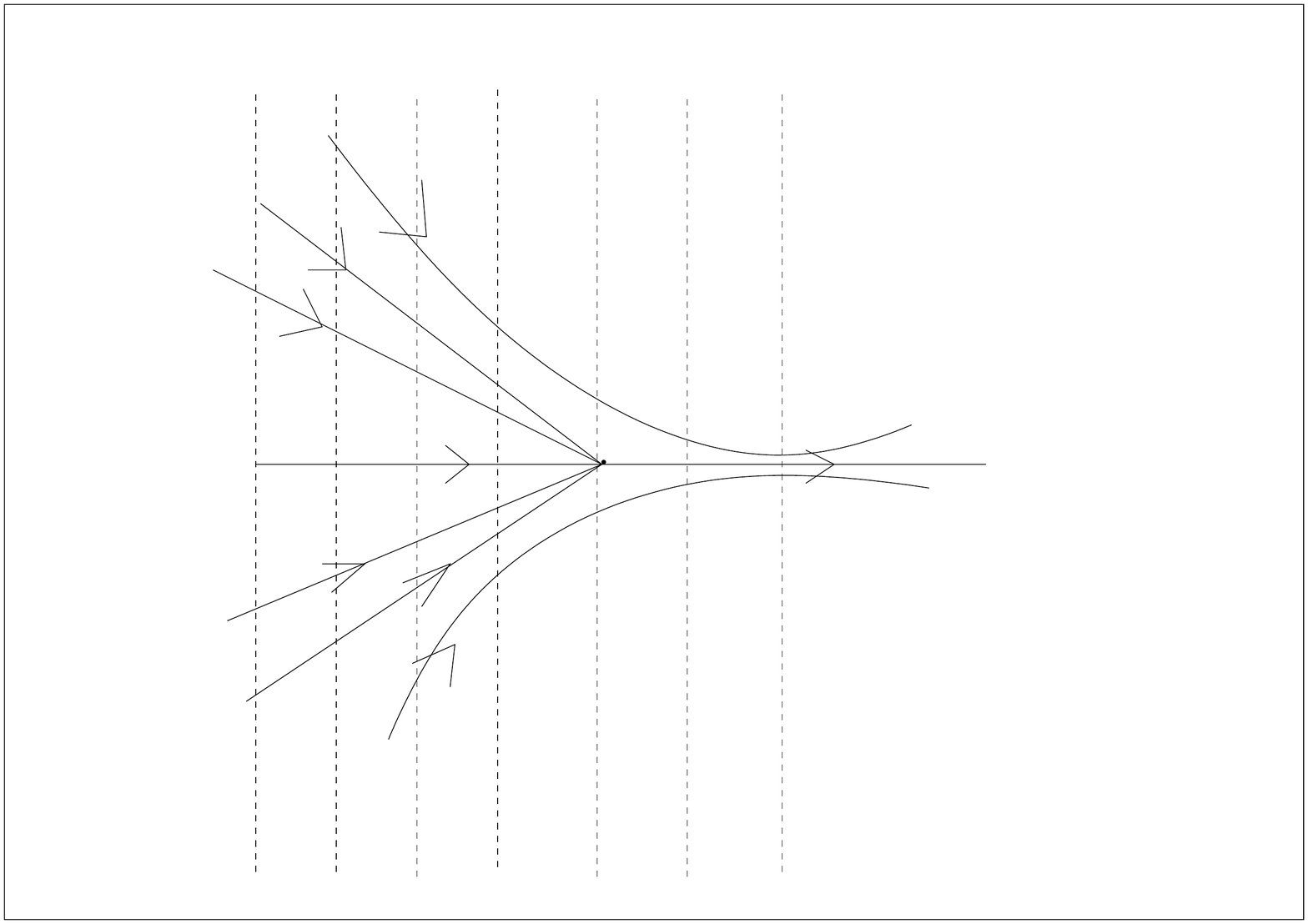}
\hfill{}
\caption{\small }
\end{figure}

\end{proof}

\begin{proposition}

\label{analytic} 

Let $f$ be an area preserving and analytic diffeomorphism of a surface $M$. Every isolated fixed point $z_0$ of $f$ of Lefschetz index $0$ is trivializable.

\end{proposition}

\begin{proof} The proposition is an immediate consequence of the results stated in the previous section. Note that in this case, there exists a continuous chart $h:U\subset M\to V\subset \R^2$  at $z_0$ such for every 
$z\in U\cap f^{-1}(U)$, one has $p_1\circ h\circ f(z)>p_1\circ h(z)$ and $p_2\circ h\circ f(z)=p_2\circ h(z)$ if $z\not=z_0$. \end{proof}

Let us state now the fundamental proposition:

\begin{proposition}

\label{fundamental}

 Suppose that $\widetilde f\in \mathrm{Diff}_0^0({\T^2})$ has only trivializable fixed points. Then,  there exists a neighborhood $\mathcal U$ of $\widetilde f$ in $\in \mathrm{Diff}_0^0({\rm T^2})$ such that $(0,0)\not\in \mathrm{int}(\rho (\widetilde{f}'))$ for every ${\widetilde f'}\in\mathcal U$.

\end{proposition}

\begin{proof} By hypothesis, $\mathrm{fix}(\widetilde f)$ is discrete and projects onto a  finite set of $\T^2$. Fix $\varepsilon_0>0$ and set $R_0(\widetilde z)= \widetilde z+ [-\varepsilon_0,\varepsilon_0]^2$ for every $\widetilde z\in \mathrm{fix}(\widetilde f)$. Conjugating $\widetilde f$ in $\mathrm{Diff}_0^0({\T^2})$ if necessary and choosing $\varepsilon _0>0$ small enough, one can suppose that the rectangles  $R_0(\widetilde z)$, $\widetilde z\in \mathrm{fix}(\widetilde f)$, are pairwise disjoint and that for every $\widetilde z\in \mathrm{fix}(\widetilde f)$ and every $z'\in R_0(\widetilde z)\setminus \{\widetilde z\}$, one has $p_1\circ \widetilde f^{-1}(\widetilde z')<p_1(\widetilde z')<p_1\circ\widetilde f(\widetilde z')$.

Fix $0<\varepsilon'_1<\varepsilon_1<\varepsilon_0$ and set 
\begin{align*}
R^+_0(\widetilde z)&= z+ (\varepsilon'_1, \varepsilon_0)\times(-\varepsilon_0,\varepsilon_0),\\
R^-_0(\widetilde z)&= z+ (-\varepsilon_0,-\varepsilon'_1)\times(-\varepsilon_0,\varepsilon_0),\\
R_1(\widetilde z)&= z+ [-\varepsilon'_1,\varepsilon'_1]\times[-\varepsilon_1,\varepsilon_1],\\
R^*_1(\widetilde z)&= z+ [-\varepsilon'_1,\varepsilon'_1]\times(-\varepsilon_1,\varepsilon_1).
\end{align*}

\begin{figure}[ht!]
\hfill
\includegraphics [height=48mm]{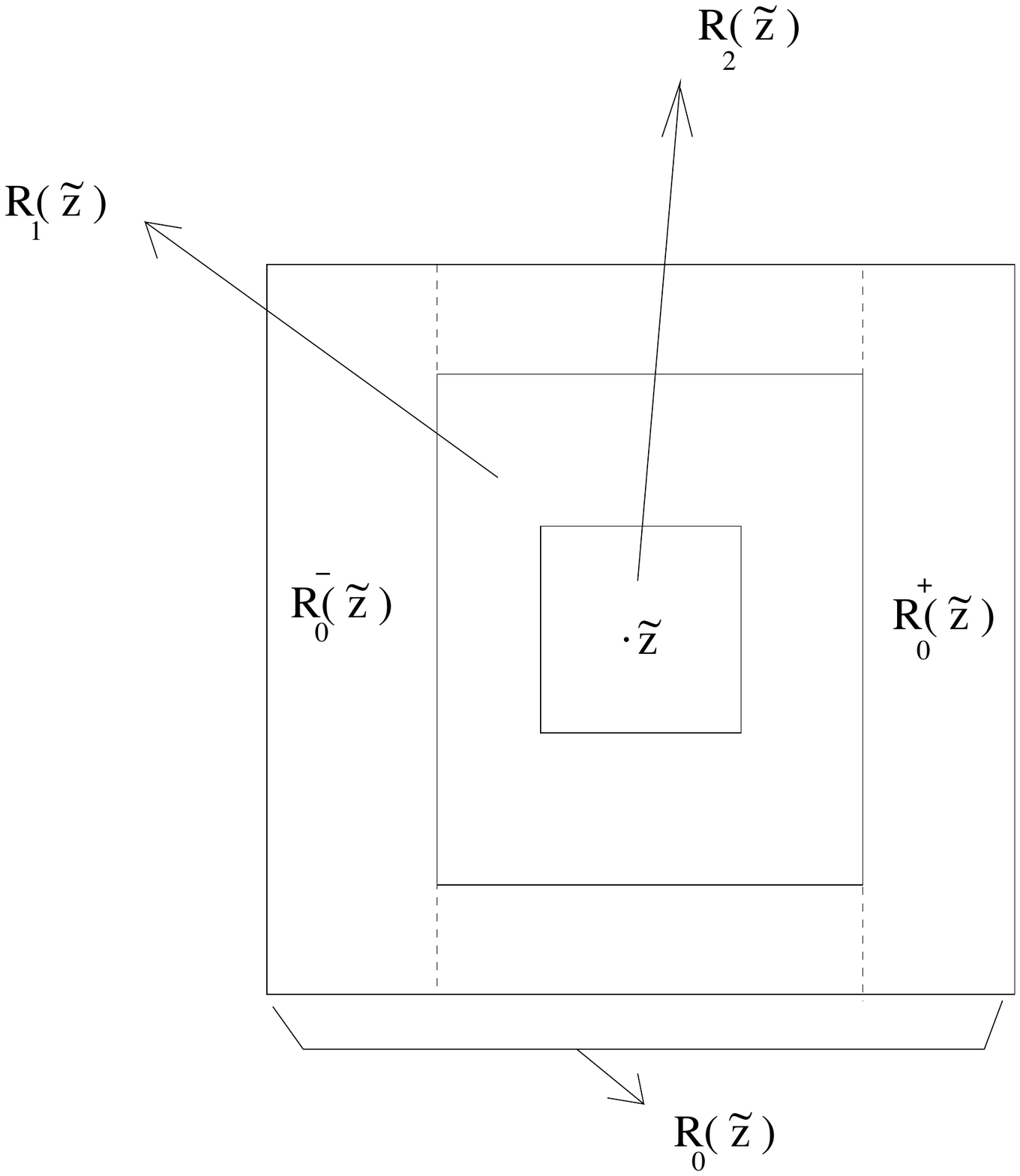}
\hfill{}
\caption{\small }
\end{figure}

We will suppose $\varepsilon_1$ small enough to ensure that for every $\widetilde z\in \mathrm{fix}(\widetilde f)$, one has
$$ f(\widetilde z+ [\varepsilon_1, \varepsilon_1]^2)\cup  f^{-1}(\widetilde z+ [\varepsilon_1, \varepsilon_1]^2 )\subset \mathrm{int}(R_0(\widetilde z)).$$This implies that if $\varepsilon'_1$ is small enough then for every $\widetilde z\in \mathrm{fix}(\widetilde f)$ one has
$$\widetilde z'\in R_1(\widetilde z)\enskip \text{and } \widetilde f(\widetilde z')\not\in R^*_1(\widetilde z)\Rightarrow f(\widetilde z')\in R^+_1(\widetilde z)$$ and
$$\widetilde z'\in R_1(\widetilde z)\enskip \text{and } \widetilde f^{-1}(\widetilde z')\not\in R^*_1(\widetilde z)\Rightarrow f^{-1}(\widetilde z')\in R^-_1(\widetilde z).$$We will suppose that this is the case.

Now fix $\varepsilon _2\in (0, \varepsilon'_1)$ and set $R_2(\widetilde z)= \widetilde z+ [-\varepsilon_2,\varepsilon_2]^2$. We will suppose $\varepsilon_2$ small enough to ensure that for every $\widetilde z\in \mathrm{fix}(\widetilde f)$ one has
$$\widetilde f(R_2(\widetilde z))\cup  \widetilde f^{-1}(R_2(\widetilde z ))\subset \mathrm{int}(R_1(\widetilde z)).$$

We will say that $\widetilde f'\in\mathrm{Diff}_0^0({\T^2})$ is a {\it positive perturbation} of $\widetilde f$ if

\begin{enumerate}

\item $\mathrm{fix}(\widetilde f')\subset  \bigcup_{\widetilde z\in \mathrm{fix}(\widetilde f)}\mathrm{int}(R_2(\widetilde z))$;

\item for all $\widetilde z\in \mathrm{fix}(\widetilde f)$ and $\widetilde z'\in R_0(\widetilde z)\setminus \mathrm{int} (R_2(\widetilde z))$, the following inequality holds $$p_1\circ \widetilde f'^{-1}(\widetilde z')<p_1(\widetilde z')<p_1\circ \widetilde f'(\widetilde z');$$

\item for all $\widetilde z\in \mathrm{fix}(\widetilde f)$ and $\widetilde z'\in R_1(\widetilde z)$ such that  $\widetilde f'(\widetilde z')\not\in R^*_1(\widetilde z)$, then $\widetilde f'(\widetilde z')\in R^+_0(\widetilde z)$;

\item for all $\widetilde z\in \mathrm{fix}(\widetilde f)$ and $\widetilde z'\in R_1(\widetilde z)$ such that  $\widetilde f'{}^{-1}(\widetilde z')\not\in R^*_1(\widetilde z)$, then $\widetilde f'^{-1}(\widetilde z')\in R^-_0(\widetilde z)$;

\item 

for every $\widetilde z\in \mathrm{fix}(\widetilde f)$, one has $\widetilde f'(R_2(\widetilde z))\cup\widetilde f'{}^{-1}(R_2(\widetilde z_2))\subset \mathrm{int}(R_1(\widetilde z))$.

\end{enumerate}

Note  that the set $\mathcal U$ of positive perturbations of $\widetilde f$ is an open neighborhood of $\widetilde f$ in $\mathrm{Diff}_0^0({\T^2})$. Of course, it contains $\widetilde f$. The fact that it is open follows from the fact that the properties (1), (2) and (5) are  open, and that the set of maps satisfying (2), (3) and (5) is open, as is the set of maps satisfying (2), (4) and (5).

 To get the proposition, we will prove that that $(0,0)\not\in \mathrm{int}(\rho (\widetilde{f}'))$ if $\widetilde f'$ is a positive perturbation of $\widetilde f$.

We will argue by contradiction, supposing that  $(0,0)\in \mathrm{int}(\rho (\widetilde{f}'))$, for a positive perturbation $\widetilde f'$ of $\widetilde f$. In that case, using a result of Franks \cite{franksrat}, one knows that there exists three periodic orbits $O_i$, $1\leq i\leq 3$, of the homeomorphism $f'$ of $\T^2$ lifted by $\widetilde f'$ such that $(0,0)$ belongs to the interior of the convex hull of the $\rho(O_i)$, $1\leq i\leq 3$. For every $z\in \pi(\mathrm{fix}(\widetilde f))$, we set $R_0(z)=\pi (R_0(\widetilde z))$ if $z=\pi(\widetilde z)$, and define similarly $R^+_0(z)$, $ R^-_0(z)$, $R_1(z)$, $R^*_1(z)$  and $ R_2(z)$.

\begin{lemma}

\label{reducing}

If the orbit $O_1$ meets a rectangle $R_1(z)$, $z\in \pi(\mathrm{fix}(f))$, then there exists a positive perturbation $\widetilde f''$ of $\widetilde f$ such that $O_2$ and $O_3$ are periodic orbits of the homeomorphism $f''$ of $\T^2$ lifted by $\widetilde f''$, with unchanged rotation vectors and  a periodic orbit $O'_1\subset O_1$ whose rotation vector is a multiple of $\rho(O_1)$ by a factor larger than $1$, such that 

\begin{itemize}

\item $\sharp (O'_1\cap R_1(z))< \sharp (O_1\cap R_1(z))$ 

\item $O'_1\cap R_1(z')=O_1\cap R_1(z')$ for every $z'\in \pi(\mathrm{fix}(f))\setminus\{z\}$.

\end{itemize}

\end{lemma}

\begin{proof} Suppose that $z\in O_1\cap R_1(z)$.  The fact that the rotation vector of $O_1$ does not vanish implies that there exist $k^-<0<k^+$ such that  $$f'{}^{k^-}(z'_1) \not\in R_1(z_0),\enskip f'{}^{k^+}(z'_1) \not\in R_1(z_0)\enskip f'{}^{k}(z'_1) \in R_1(z_0) \text{ if }k^-<k<k^+.
$$  One deduces that $f'{}^{k^-}(z) \in R^-_0(z)$ and $f'{}^{k^+}(z) \in R^+_0(z)$ because $\widetilde f'$ is a positive perturbation of $\widetilde f$. One can find a horizontal graph $\Gamma$ in $\mathrm{int}(R_0(z))\setminus R_1(z)$ joining $f'{}^{k^-}(z)$ to $f'{}^{k^+}(z)$, which means a simple path that projects injectively onto the first factor of $\T^2$, then a neighborhood $U\subset \mathrm{int}(R_0(z))\setminus R_1(z)$ of $\Gamma$. The graph and the neighborhood can be chosen to meet $O_1\cup O_2\cup O_3$ only at the points $f'{}^{k^-}(z)$ and $f'{}^{k^+}(z) $. One can find a homeomorphism $h$ supported on $U$ such that $h(f'{}^{k^-}(z)) =f'{}^{k^+}(z)$. Moreover $h$ may be chosen such that it is lifted to a homeomorphism $\widetilde h\in\mathrm{Diff}_0^0({\T^2})$ supported on $\pi^{-1}(U)$  and satisfying $p_1\circ\widetilde h(\widetilde z')\geq p_1(\widetilde z')$ for every $\widetilde z'\in\R^2$.

\begin{figure}[ht!]
\hfill
\includegraphics [height=48mm]{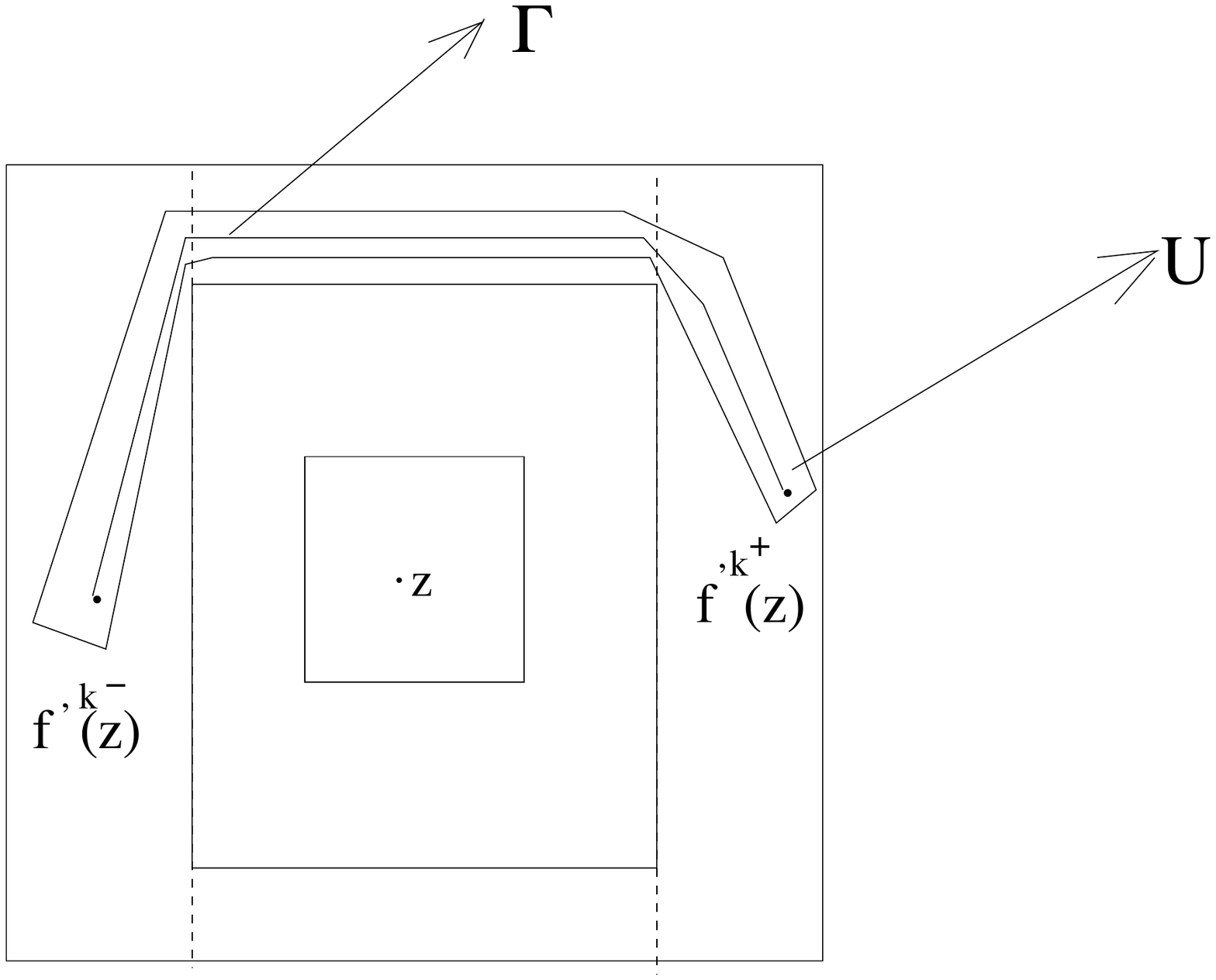}
\hfill{}
\caption{\small }
\end{figure}

Let us explain why $\widetilde f''=\widetilde h\circ\widetilde f$ is a positive pertubation of $\widetilde f$ by   verifying that the five properties of the definition of positive perturbation are satisfied by $\widetilde f'''$

\smallskip
\noindent  (2) \enskip If there exists $\widetilde z\in \mathrm{fix}(\widetilde f)$ such that $\widetilde z'\in R_0(\widetilde z)\setminus \mathrm{int} (R_2(\widetilde z))$, then $\widetilde h^{-1}(\widetilde z') \in R_0(\widetilde z)\setminus \mathrm{int} (R_2(\widetilde z))$ and so 
$$p_1\circ\widetilde f'^{-1}\circ \widetilde h^{-1}(\widetilde z')<p_1\circ\widetilde h^{-1}(\widetilde z')\leq p_1(\widetilde z')<p_1\circ\widetilde f'(\widetilde z')\leq p_1\circ\widetilde h\circ \widetilde f'(\widetilde z').$$ 

\smallskip

\noindent  (1) \enskip The maps $\widetilde f''^{-1}$ and $\widetilde f'^{-1}$ coincide on the complement of $\bigcup_{\widetilde z\in \mathrm{fix}(\widetilde f)}R_0(\widetilde z)$ and the last one has no fixed point in this complement, so the fixed point set of $\widetilde f''$ is included in $\bigcup_{\widetilde z\in \mathrm{fix}(\widetilde f)}R_0(\widetilde z)$. The map $\widetilde f''$ satisfying $2$, its fixed point set is included in $\bigcup_{\widetilde z\in \mathrm{fix}(\widetilde f)}\mathrm{int}(R_2(\widetilde z))$.

\smallskip
\noindent  (3) \enskip If there exists $\widetilde z\in \mathrm{fix}(\widetilde f)$ such that $\widetilde z'\in R_1(\widetilde z)$ and  $\widetilde f''(\widetilde z')\not\in R^*_1(\widetilde z)$, then $\widetilde f'(\widetilde z')\not\in R^*_1(\widetilde z)$ (otherwise $\widetilde f''(\widetilde z')=\widetilde f'(\widetilde z')$) and so $\widetilde f'(\widetilde z')\in R^+_0(\widetilde z)$, which implies that $\widetilde f''(\widetilde z')\in R^+_0(\widetilde z)$, because 
$\widetilde h(R^+_0(\widetilde z))\subset R^+_0(\widetilde z)$.

\smallskip
\noindent  (4) \enskip   If there exists $\widetilde z\in \mathrm{fix}(\widetilde f)$ such that $\widetilde z'\in R_1(\widetilde z)$ and  $\widetilde f''{}^{-1}(\widetilde z')\not\in R^*_1(\widetilde z)$, then $\widetilde f'{}^{-1}(\widetilde z')=\widetilde f''^{-1}(\widetilde z')\not\in R^*_1(\widetilde z)$, and so $\widetilde f''{}^{-1}(\widetilde z')=\widetilde f'^{-1}(\widetilde z') \in R^-_0(\widetilde z)$.

\smallskip
\noindent  (5) \enskip  For every $\widetilde z\in \mathrm{fix}(\widetilde f)$, one has $\widetilde f'(R_2(\widetilde z))\cup\widetilde f'{}^{-1}(R_2(\widetilde z_2))\subset \mathrm{int}(R_1(\widetilde z))$. It implies that $\widetilde f''(R_2(\widetilde z))=\widetilde f'(R_2(\widetilde z))$ and $\widetilde f''{}^{-1}(R_2(\widetilde z_2))=\widetilde f'{}^{-1}(R_2(\widetilde z_2))$ and consequently that 
$\widetilde f''(R_2(\widetilde z))\cup\widetilde f''{}^{-1}(R_2(\widetilde z_2))\subset \mathrm{int}(R_1(\widetilde z))$.

Observe now that $f'$ and $f''$ coincide on $O_2$ and $O_3$ and that $\widetilde f'$ and $\widetilde f''$ coincide on $\pi^{-1}(O_2)$ and $\pi^{-1}(O_3)$. So, $O_2$ and $O_3$ are periodic orbits of $f''$ with rotation vectors unchanged. The orbit $O_1$ has been replaced by a periodic orbit $O'_1\subset O_1$ of shorter period whose rotation vector (for the lift $\widetilde f''$) is a multiple of the rotation vector of $O_1$ (for $\widetilde f'$) by a factor larger than $1$.\end{proof}

Note that $(0,0)$ is still in the interior of the convex hull of the new rotation vectors, the old ones multiplied by numbers greater than $1$. Applying the lemma  finitely many times to each orbit $O_i$ and each rectangle $R_1(z)$, one can always suppose that the orbits $O_i$ of $f'$ do not meet the rectangles $R_1(z)$, $z\in \pi(\mathrm{fix}(\widetilde f))$. One can find  $\widetilde h\in\mathrm{Diff}_0^0({\T^2})$ supported on $ \bigcup_{\widetilde z\in \mathrm{fix}(\widetilde f)}R_1(\widetilde z)$, that satisfies
$p_1\circ\widetilde h(\widetilde z')\geq p_1(\widetilde z')$ for every $\widetilde z'\in\R^2$ and such that
$\widetilde h\circ\widetilde f' (R_2(\widetilde z))\cap R_2(\widetilde z)=\emptyset$ for every $\widetilde z\in \mathrm{fix}(\widetilde f)$.

Note that $\widetilde g=\widetilde h\circ \widetilde f'$ is fixed point free. Indeed:

\begin{itemize}

\item $\widetilde g^{-1}$ and $\widetilde f'^{-1}$ coincide on the complement of $\bigcup_{\widetilde z\in \mathrm{fix}(\widetilde f)}R_1(\widetilde z)$ and $\widetilde f'^{-1}$ has no fixed point in this complement, so it is the same for $\widetilde g^{-1}$;

\item If there exists $\widetilde z\in \mathrm{fix}(\widetilde f)$ such that $\widetilde z'\in R_1(\widetilde z)\setminus \mathrm{int} (R_2(\widetilde z))$, then $p_1\circ\widetilde h\circ\widetilde f'(\widetilde z')\geq p_1\circ \widetilde f'(\widetilde z')>p_1(\widetilde z')$; 

\item  $\widetilde g(R_2(\widetilde z))\cap R_2(\widetilde z)=\emptyset$ for every $\widetilde z\in \mathrm{fix}(\widetilde f)$.

\end{itemize}

On the other hand,  each $O_i$ is a periodic orbit of the homeomorphism $g$ of $\T^2$ lifted by $\widetilde g$, with unchanged rotation vector because $O_i$ is disjoint from $\bigcup_{\widetilde z\in \mathrm{fix}(\widetilde f)}R_1(\widetilde z)$. In particular $(0,0)$ belongs to the interior of $\rho(\widetilde g)$. This contradicts Franks' result.
 \end{proof}

\section{Proofs of the theorems}

\subsection{Proof of Theorem 1}

Set $\rho=(p_1/q,p_2/q)$, where $p_1$, $p_2$ and $q$ are relatively prime. Replacing $\widetilde f$ by $\widetilde f^{q}-(p_1,p_2)$, it is sufficient to study the case where $\rho=(0,0)$. Denote $f$ the homeomorphism of $\T^2$ lifted by $\widetilde f$.

The function
$$\widetilde g:\widetilde z\mapsto  (p_1\circ \widetilde f(\widetilde z)-p_1(\widetilde z))^2+ (p_2\circ \widetilde f(\widetilde z)-p_2(\widetilde z))^2$$lifts an analytic function $g$ on $\T^2$ that vanishes exactly on $\pi(\mathrm{fix}(\widetilde f))$. If this set is not finite, it contains a simple closed curve (see \cite{non2} or \cite{milnor} ). Such a curve must be homotopically trivial because the rotation set of $\widetilde f$ has interior. This means that $\mathrm{fix}(\widetilde f)$ contains a simple closed curve $\Gamma$. This curves bounds a topological open disk $D$ invariant by $\widetilde f$. Moreover $\widetilde f$ is not the identity on this disk because it is analytic and not equal to the identity on the whole plane (its rotation set is not trivial). The fact that $\widetilde f\vert_D$ is area preserving implies, by a classical consequence of Brouwer's theory (see Franks \cite{f2} for example) that there exists a closed curve $C\subset D$ such that the Lefschetz index of $\widetilde f$ on $C$ is 1. Such a property, and consequently the existence of a fixed point is still satisfied for all $C^0$ perturbations of $\widetilde f$. In particular,  there exists a neighborhood $\mathcal U$ of $\widetilde f$ in $ \mathrm{Diff}_0^0({\rm T^2})$ such that every ${\widetilde f'}\in\mathcal U$ has a fixed point  and consequently one has $(0,0)\in \rho (\widetilde{f}')$. Suppose now that $\pi(\mathrm{fix}(\widetilde f))$ is finite. If the Lefchetz index of one fixed point of $\widetilde f$ is non zero, then there exists a neighborhood $\mathcal U$ of $\widetilde f$ in $ \mathrm{Diff}_0^0({\rm T^2})$ such that every ${\widetilde f'}\in\mathcal U$ has a fixed point  and consequently one has $(0,0)\in \rho (\widetilde{f}')$. It remains to study the case where all indices are equal to zero. Proposition \ref{analytic} tells us that every fixed point is trivializable. Applying Proposition \ref{fundamental}, one deduces that there exists a neighborhood $\mathcal U$ of $\widetilde f$ in $\in \mathrm{Diff}_0^0({\rm T^2})$ such that $(0,0)\not\in \mathrm{int}(\rho (\widetilde{f}'))$ for all ${\widetilde f'}\in\mathcal U$.

\subsection{Proof of Theorem 2}

Suppose that $\widetilde f$ is not highly degenerate at $\rho=(p/q,r/q)$ and that there exists a periodic point $z$  of period $q$ and rotation vector $\rho$ such that $1$ is not an eigenvalue of $Df^q(z)$. In that case the Lefschetz index of $f^q$ at $z$ is different from $0$ and one can conclude that there exists a neighborhood $\mathcal U$ of $\widetilde f$ in $\widetilde{\mathrm{Diff}}_0^0({\T^2})$ such that $\rho\in \rho (\widetilde{f}')$ for every ${\widetilde f'}\in\mathcal U$. If $1$ is an eigenvalue of $Df^q(z)$, for every point $z$  of period $q$ and rotation vector $\rho$, then every such point is trivializable by Propoistion \ref{saddle-node}. Applying Proposition \ref{fundamental}, one deduces that there exists a neighborhood $\mathcal U$ of $\widetilde f$ in $\widetilde{\mathrm{Diff}}_0^0({\T^2})$ such that $\rho\not\in \mathrm{int}(\rho (\widetilde{f}'))$ for every ${\widetilde f'}\in\mathcal U$.

\subsection{Proof of Theorem 3}

The main properties of the map $\widetilde f\in \widetilde{\mathrm{Diff}}^{\infty}_0(\T^2)$ that we want to construct are the following ones:

\begin{itemize}

\item the rotation set of $\widetilde f$ is included in the non negative cone $[0,+\infty)^2$ and contains the vectors $(0,0)$, $(0,1)$ and $(1,0)$;

\item each vector $(0,0)$, $(0,1)$ and $(1,0)$ is the rotation vector of a fixed point of the diffeomorphism $f\in \mathrm{Diff}_0^{\infty}(\T^2)$ lifted by $\widetilde f$;

\item the unique fixed point of rotation vector $(0,0)$ is $(0,0)+\Z^2$, it has a homoclinic point lifted to a heteroclinic point from $(0,1)$ to $(0,0)$ and a homoclinic point lifted to a heteroclinic point from $(1,0)$ to $(0,0)$;

\item the vector field $\widetilde z\mapsto \widetilde f(\widetilde z)-\widetilde z$ has no values in the negative cone $(-\infty,0)^2$;

\item each vertical line $\{k\}\times\R$, $k\in\Z$, is sent on its right by $\widetilde f$ and each horizontall line $\R\times\{k\}$ sent above.

\end{itemize}

The third assertion allows us to perturb $\widetilde f$ in a way that the new map has a rotation set whose interior contains $(0,0)$ and the last two assertions allow us to perturb $\widetilde f$ in a way that the new map has a rotation set which does not contain $(0,0)$.

The sets
$$\left\{(x,y)\in\R^2\,\vert\, \vert y\vert \leq \frac{1}{2\pi} \vert \sin (\pi x)\vert\right\}, \enskip \left\{(x,y)\in\R^2\,\vert\, \vert x\vert \leq \frac{1}{2\pi} \vert \sin (\pi y)\vert\right\}$$
project by $\pi$ onto connected compact subsets of $\T^2$ respectively denoted $H$ and $V$. We set $C=\overline{\T^2\setminus (H\cup V)}$ and then define $$\widetilde H=\pi^{-1}(H), \enskip \widetilde V=\pi^{-1}(V),  \enskip \widetilde C=\pi^{-1}(C).$$

\begin{proposition} \label{ExampleDiss} There exists $\widetilde{f}_H\in \widetilde{\mathrm{Diff}}^{\infty}_0(\T^2)$ such that:

\begin{enumerate}

\item  the fixed point set of $\widetilde{f}_H$ is $\widetilde H$;

\item  for every $\widetilde{z}\in \widetilde V\cup \widetilde C$, one has $p_1\circ \widetilde{f}_H(\widetilde{z})\geq p_1(\widetilde{z})$;

\item for every $\widetilde{z}\in  \mathrm{int}(\widetilde C)$, one has $p_1\circ \widetilde{f}_H(\widetilde{z})>p_1(\widetilde{z})$ and $p_2\circ \widetilde{f}_H(\widetilde{z})=p_2(\widetilde{z});$

\item  there exists $\widetilde{z}_0\in\widetilde V$ such that 
$$\lim_{k\to -\infty}\widetilde{f}_H^k(\widetilde{z}_0)=(0,1), \enskip \lim_{k\to +\infty}\widetilde{f}_H^k(\widetilde{z}_0)=(0,0);$$

\item  there exists  $\widetilde z_1\in \R\times\{1/2\}$ such that $\widetilde{f}_H(\widetilde{z}_1)=\widetilde{z}_1+(1,0).$

\end{enumerate}

\end{proposition}

\begin{proof}

 Let us begin by choosing a pair of smooth $\Z^2$-periodic  real valued functions $\xi$ and $\eta$ on $\R^2$ such that:

\begin{itemize}

\item  $\xi$ vanishes on $\widetilde H\cup (\Z\times\R)$ and is positive elsewhere;

\item   $\eta$ vanishes on $\widetilde H\cup \widetilde C$ and is negative elsewhere (which means on $\mathrm{int}(\widetilde V)$).

\end{itemize}

The map $\widetilde{f}_\epsilon:  \widetilde{z}\mapsto \widetilde{z}+\varepsilon .(\xi(\widetilde{z}),\eta(\widetilde{z}))$ is smooth and lifts a smooth transformation of $\T^2$ homotopic to the identity. As the set of $C^\infty $ diffeomorphisms of the torus is open, if $\epsilon>0$ is sufficiently small, then $\widetilde{f}_{\epsilon}$ belongs to $\widetilde{\mathrm{Diff}}^{\infty}_0(\T^2)$.  The fixed point set of $\widetilde f_{\varepsilon}$ is $\widetilde H$. Moreover $\widetilde f_{\varepsilon}$ fixes every vertical $\{k\}\times\R$, $k\in\Z$, moving every point $(k,y)$, $y\not\in\Z$,negatively in the vertical direction. 

So we can choose a point $\widetilde z_0\in \{0\}\times (0,1)$ whose orbit avoids $\R\times\{1/2\}$ and its $\alpha$-limit is $(0,1)$ and its $\omega$-limit is $(0,0).$


The point $\widetilde z_1=(1/2,1/2)\in \partial \widetilde V$ is sent on its right by $\widetilde f_{\varepsilon}$ still on the horizontal line $\R\times\{1/2\}$. Let us choose $\delta\in (0, 1/2-1/2\pi)$ such that the orbit of $\widetilde z_0$ avoids $\R\times[1/2-\delta, 1/2+\delta]$. One can construct $\widetilde f'\in \widetilde{\mathrm{Diff}}^{\infty}_0(\T^2)$ such that:

\begin{itemize}

\item  $\widetilde f'(\widetilde z)=\widetilde z$ if $\delta\leq\vert p_2(\widetilde z)-1/2\vert\leq 1/2$;

\item  $p_2\circ \widetilde f' (\widetilde z)=p_2(\widetilde z)$ for all $\widetilde z\in \R^2$;

\item  $p_1\circ \widetilde f' (\widetilde z)\geq p_1(\widetilde z)$ for all $\widetilde z\in \R^2$;

\item $\widetilde f'\circ \widetilde{f}_{\epsilon} (\widetilde{z}_1)=\widetilde{z}_1+(1,0).$

\end{itemize} 

Let us verify that $\widetilde f_H=\widetilde f'\circ \widetilde f_{\varepsilon}$ satisfies the properties formulated in the proposition.

The assertion {(5)} is satisfied by construction and {(4)} because the orbit of $\widetilde z_0$ avoids the support of $\widetilde f'$. To get {(3)} note that $p_1\circ \widetilde{f}_H(\widetilde{z})\geq p_1\circ \widetilde{f}_{\varepsilon}(\widetilde{z})>p_1(\widetilde{z})$ and $p_2\circ \widetilde{f}_H(\widetilde{z})=p_2\circ \widetilde{f}_{\varepsilon}(\widetilde{z})=p_2(\widetilde{z})$ for every $\widetilde{z}\in  \mathrm{int}(\widetilde C)$. To get {(2)} note that $p_1\circ \widetilde{f}_H(\widetilde{z})\geq p_1\circ \widetilde{f}_{\varepsilon}(\widetilde{z})\geq p_1(\widetilde{z})$ for every $\widetilde{z}\in \widetilde V\cup \widetilde C$.  In fact the last inequality is strict if moreover $\widetilde z\not\in\Z\times\R$ and we have $p_2\circ \widetilde{f}_H(\widetilde{z})=p_2\circ \widetilde{f}_{\varepsilon}(\widetilde{z})< p_2(\widetilde{z})$ if $\widetilde z\in\Z\times\R$. Consequently the fixed point set of $\widetilde f_H$ is included in $\widetilde H$. Conversely, the fact that $\delta<1/2-1/2\pi$ implies that $\widetilde H$ is included in the fixed point set of $\widetilde f_H$, so {(1)} is proved.\end{proof}

In the same way, we can prove:

\begin{proposition} There exists $\widetilde{f}_V\in \widetilde{\mathrm{Diff}}^{\infty}_0(\T^2)$ such that:

\begin{enumerate}

\item  the fixed point set of $\widetilde{f}_V$ is $\widetilde V$;

\item  for every $\widetilde{z}\in \widetilde H\cup \widetilde C$, one has $p_2\circ \widetilde{f}_V(\widetilde{z})\geq p_2(\widetilde{z})$;

\item for every $\widetilde{z}\in  \mathrm{int}(\widetilde C)$, one has $p_2\circ \widetilde{f}_V(\widetilde{z})>p_2(\widetilde{z})$ and $p_1\circ \widetilde{f}_V(\widetilde{z})=p_1(\widetilde{z});$

\item  there exists $\widetilde{z}'_0\in\R^2$ such that
$$\lim_{k\to -\infty}\widetilde{f}_V^k(\widetilde{z}'_0)=(1,0), \enskip \lim_{k\to +\infty}\widetilde{f}_V^k(\widetilde{z}'_0)=(0,0);$$

\item  there exists  $\widetilde z'_1\in\{1/2\}\times\R$ such that $\widetilde{f}_V(\widetilde{z}'_1)=\widetilde{z}'_1+(0,1).$

\end{enumerate}

\end{proposition}

It remains to prove that $\widetilde f=\widetilde f_V\circ \widetilde f_H$ satisfies the properties formulated in Theorem 3. Let us study the properties of the vector field $z\mapsto \widetilde f(\widetilde z)-\widetilde z$:

\begin{itemize}

\item if $\widetilde z\in \mathrm{int}(\widetilde C)$ and  $\widetilde f_H(\widetilde z)\in \mathrm{int}(\widetilde C)$, then 
$$p_1\circ \widetilde{f}(\widetilde{z})=p_1\circ \widetilde{f}_H(\widetilde{z})>p_1(\widetilde{z})\enskip\mathrm{and}\enskip p_2\circ \widetilde{f}(\widetilde{z})>p_2\circ \widetilde{f}_H(\widetilde{z})=p_2(\widetilde{z});$$

\item if $\widetilde z\in \mathrm{int}(\widetilde C)$ and  $\widetilde f_H(\widetilde z)\in \widetilde V$, then
$$p_1\circ \widetilde{f}(\widetilde{z})=p_1\circ \widetilde{f}_H(\widetilde{z})>p_1(\widetilde{z})\enskip\mathrm{and}\enskip p_2\circ \widetilde{f}(\widetilde{z})=p_2\circ \widetilde{f}_H(\widetilde{z})=p_2(\widetilde{z});$$

\item if $\widetilde z\in \widetilde V$ and  $\widetilde f_H(\widetilde z)\in \mathrm{int}(\widetilde C))$, then $p_1\circ \widetilde{f}(\widetilde{z})=p_1\circ \widetilde{f}_H(\widetilde{z})>p_1(\widetilde{z})$;

\item if $\widetilde z\in \widetilde V$ and $\widetilde f_H(\widetilde z)\in \widetilde V$, then $p_1\circ \widetilde{f}(\widetilde{z})=p_1\circ \widetilde{f}_H(\widetilde{z})\leq p_1(\widetilde{z})$ and $\widetilde f(z)\not=0$ if $z\not\in\Z^2$;

\item if $\widetilde z\in \widetilde H$, then $p_2\circ \widetilde{f}(\widetilde{z})\geq p_2\circ \widetilde{f}_H(\widetilde{z})= p_2(\widetilde z)$ and $\widetilde f(z)\not=0$ if $z\not\in\Z^2$.

\end{itemize}

Summarizing, the vector field $z\mapsto \widetilde f(\widetilde z)-\widetilde z$ vanishes only on $\Z^2$ and takes its value out of the negative cone $(-\infty,0)^2$. Moreover each vertical line $\{k\}\times\R$ is sent on its right by $\widetilde f$ and each horizontal line $\R\times\{k\}$ is sent above. One deduces that the rotation set of $\widetilde f$ is included in the non negative cone $[0,+\infty)^2$.  Let us consider a vector $v$ in the negative cone. The properties about vertical and horizontal lines are still satisfied for the diffeomorphism $\widetilde f-v$. So the rotation set of $\widetilde f-v$ is contained in the non negative cone.  But $\widetilde f-v$ is fixed point free,  and so by a result of J. Franks \cite{f2}, $(0,0)$ cannot be an extremal point of the rotation set. Consequently, it does not belong to this set. The vector $v$ may be chosen arbitrarily small, so one can perturb $\widetilde f$ in a way that $(0,0)$ does not belong to the rotation set.

By construction, one knows that $\widetilde f(\widetilde z_1)=\widetilde f_H(\widetilde z_1)= z_1+(1,0)$ and  $\widetilde f(\widetilde z'_1)=\widetilde f_V(\widetilde z'_1)= z'_0+(0,1)$. So the rotation set of $\widetilde f$ contains $(0,0)$, $(1,0)$ and $(0,1)$ and has non empty interior. Similarly, one knows $$\lim_{k\to -\infty}\widetilde{f}^k(\widetilde{z}_0)=(0,1), \enskip \lim_{k\to +\infty}\widetilde{f}^k(\widetilde{z}_0)=(0,0)$$ and $$\lim_{k\to -\infty}\widetilde{f}^k(\widetilde{z}'_0)=(1,0), \enskip \lim_{k\to +\infty}\widetilde{f}^k(\widetilde{z}'_0)=(0,0).$$

So, one can perturb $\widetilde f$  in a neighborhood of $\Z^2$ for the $C^0$ topology and get a map $\widetilde f'\in \widetilde{\mathrm{Diff}}^{\infty}_0(\T^2)$ arbitrarily close to $\widetilde f$ such that 
$$\widetilde f'(\widetilde z_1)=z_1+(1,0), \enskip \widetilde f(\widetilde z'_1)=z'_1+(0,1)$$ and such that there exist positive integers  $q$ and $q'$ satisfying
$$\widetilde f'^{q}(\widetilde z_0)=z_0-(0,1), \enskip \widetilde f'^{q'}(\widetilde z'_1)=z'_1-(1,0).$$

The rotation set of $\widetilde f'$ contains the vectors $(0,1)$, $(1,0)$, $(-1/q, 0)$, $(0, -1/q')$, so its interior contains $(0,0)$.

\subsection{Proof of Theorem 4}

The proof is very similar to the proof of Theorem 3, writing the example as a composition of a ``horizontal map'' and a ``vertical map''. We want the maps to satisfy similar properties but to be area preserving. The main problem is the construction of the homoclinic points, essential in the proof. As we will see, it is possible to do it, working in the space of homeomorphisms. Of course, one may ask if there exists a smooth or even a differentiable example (recall that Theorem \ref{main1} tells us that there is no in the space of analytic diffeomorphisms).

We will begin by changing the definition of the set $H$, $V$ and $C$. Fix $\alpha\in(0,1/10]$. The sets
$$\left\{(x,y)\in[-\frac 12,\frac 12]^2\,\vert\, \vert y\vert \leq 2\alpha \vert x\vert \right\}, \enskip \left\{(x,y)\in[-\frac 12,\frac 12]^2\,\vert\, \vert x\vert \leq 2\alpha \vert y\vert \right\}$$
project by $\pi$ onto connected compact subsets of $\T^2$ respectively denoted $H$ and $V$. We set $C=\overline{\T^2\setminus (H\cup V)}$ and then define $$\widetilde H=\pi^{-1}(H), \enskip \widetilde V=\pi^{-1}(V),  \enskip \widetilde C=\pi^{-1}(C).$$

The analogous of Proposition \ref{ExampleDiss} is the following:

\begin{proposition}  \label{ExampleConservative}There exists $\widetilde{f}_H\in \widetilde{\mathrm{Diff}}^{0}_0(\T^2)$ such that:

\begin{enumerate}

\item $f$ is area preserving

\item  the fixed point set of $\widetilde{f}_H$ is $\widetilde H$;

\item if $\widetilde{z}\in \widetilde C\setminus\widetilde H$, the vector $\widetilde f_H(\widetilde z)-\widetilde z$ belongs to the cone of equation $\vert y\vert <x$;

\item if $\widetilde{z}\in \widetilde V\setminus\widetilde H$ and $\widetilde f(\widetilde{z})\not\in \widetilde V$, the vector $\widetilde f_H(\widetilde z)-\widetilde z$ belongs to the cone of equation $\vert y\vert <x$;

\item if $\widetilde{z}\in \widetilde V\setminus\widetilde H$ and $\widetilde f(\widetilde{z})\in \widetilde V$, the vector $\widetilde f_H(\widetilde z)-\widetilde z$ belongs to the half-plane of equation $y<x$;

\item  there exists $\widetilde{z}_0\in\widetilde V$ such that $$\lim_{k\to -\infty}\widetilde{f}_H^k(\widetilde{z}_0)=(0,1), \enskip \lim_{k\to +\infty}\widetilde{f}_H^k(\widetilde{z}_0)=(0,0);$$

\item  there exists  $\widetilde z_1\in \widetilde V$ such that $\widetilde{f}_H(\widetilde{z}_1)=\widetilde{z}_1+(1,0).$

\end{enumerate}

\end{proposition}

\begin{proof} We will construct $\widetilde f_H$ step by step. Let us begin by stating elementary facts.

There exists a diffeomorphism $\theta :[0,1]\to [\alpha , 1-\alpha]$, uniquely defined, such that, for every $y\in[0,1]$ the quadrilater joining the points 
$$( \alpha,1/2), \,(1/2-\alpha, 1/2),  \,(\alpha -\alpha\vert 2y-1\vert,y), \,(1/2-\alpha, y)$$ and the quadrilater joining the points 
$$( \alpha,1/2), \,(1/2, 1/2), \,(\alpha -\alpha\vert 2y-1\vert,y),\,(1/2, \theta(y)) $$ have the same area. Consequently, there exists a continuous map $\widetilde h:\R\times[0,1]\to\R^2$, uniquely defined, such that

\begin{itemize}

\item the image of $\widetilde h$ is equal to $(\widetilde V\cup\widetilde C)\cap \R\times[0,1]$;

\item the map $\widetilde h$ is a homeomorphism between $\R\times[0,1]$ and its image;

\item the map $\widetilde h$ preserves the area

\item the image of the horizontal segment $$[\alpha -\alpha\vert 2y-1\vert, 1-3\alpha +\alpha\vert y-1/2\vert]\times\{y\}$$ is the broken segment passing through $$(\alpha -\alpha\vert 2y-1\vert,y), \, (1/2, \theta(y)) \enskip \mathrm{and} \enskip (1-\alpha +\alpha\vert 2y-1\vert,y);$$

\item $\widetilde h(x+k(1-2\alpha), y)=\widetilde h(x,y)+(k,0)$, for every $(x,y)\in \R\times[0,1]$ and every $k\in\Z$;

\item the fixed point set of $\widetilde h$ is equal to $\widetilde V\cap [-1/2,1/2]\times[0,1]$.

\end{itemize}

One gets an area preserving homeomorphism of $\R\times[0,1]$ by setting
$$\widetilde f_0(x,y)= (x+\alpha/2 -\alpha\vert y-1/2\vert, y)$$ whose fixed point set is the boundary of the strip $\R\times[0,1]$. Consequently, the map $\widetilde h\circ \widetilde f_0 \circ \widetilde h^{-1}$ can be extended in a unique way to a map $\widetilde f_1\in \widetilde{\mathrm{Diff}}^0_0(\T^2)$ whose fixed point set is equal to $\widetilde H$. Moreover $\widetilde f$ is area preserving and coincides  with $\widetilde f_0$ on the quadrilater
$$\{(x,y)\in\R\times[0,1]\enskip\vert\enskip \alpha -\alpha\vert 2y-1\vert\leq x\leq \alpha/2 -\alpha\vert y-1/2\vert\}.$$

Note that the slope of the segment joining $(\alpha -\alpha\vert 2y-1\vert,y)$ to $(1/2, \theta(y))$ and the slope of the segment joining $(1/2, \theta(y))$ to $(1-\alpha +\alpha\vert 2y-1\vert,y)$ are both smaller than $1$ in modulus. This implies  that if $\widetilde{z}\not\in\widetilde H$, then $\widetilde f_1(\widetilde z)-\widetilde z$ belongs to the cone of equation $\vert y\vert <x$;

Now, fix $\beta\in (0,1)$ and define a sequence $(y_k)_{k\in\Z}$ by the inductive relation
$$y_0=1/2, \enskip (1-y_{k-1}) =\beta(1-y_k) \enskip \mathrm{if} \enskip k\leq 0, \enskip y_{k+1}=\beta y_k \enskip \mathrm{if} \enskip k\geq 0.$$ 
The sequence $(y_k)_{k\in\Z}$ is decreasing and satisfies $$\lim_{k\to -\infty} y_k=1,\enskip \lim_{k\to+\infty }y_k=0.$$
Denote $\gamma_k$ the segment joining $\widetilde f_1(0, y_k)$ to $(0,y_{k+1})$. The segments $(\gamma_k)_{k\in\Z}$ are pairwise disjoint and each $\widetilde f_1^{-1}(\gamma_k)$ is a segment disjoint from $\gamma_k$. Note that every vector $z-z'$, $z\in\gamma_k$, $z'\in \widetilde f_1^{-1}(\gamma_k)$ belongs to the half-plane of equation $y<x$. So, one can construct a family of pairwise disjoint topological open disks $(D_k)_{k\in\Z}$ such that $D_k$ is a neighborhood of $\gamma_k$ disjoint from $\widetilde f_1^{-1}(D_k)$ and such that every vector $z-z'$, $z\in D_k$, $z'\in \widetilde f_1^{-1}(D_k)$ belongs to the half-plane of equation $y<x$.

For every $k\in\Z$, one can find an area preserving homeomorphism $g_k$ supported on $D_k$ and sending the point  $\widetilde f_1(0, y_k)$ onto $(0,y_{k+1})$. There exists a unique element $\widetilde g$ of $\widetilde{\mathrm{Diff}}^0_0(\T^2)$ that for every $k\in\Z$ coincides with $g_k$ on $D_k$ and fixes every point with no $\Z^2$ translation in a $D_k$.  This homeomorphism is area preserving. Note that $\widetilde f_2=\widetilde g\circ\widetilde f_1$ satisfies all properties formulated in Proposition \ref{ExampleConservative} but the last one. A last step is necessary to obtain $\widetilde f_H$.

Let us choose $1/2<a<c<b<\min (y_{-1}, 1-\alpha)$ and a continuous map $\varphi:[0,1]\to\R$ that vanishes out of $(a,b)$, takes positive values in $(a,b)$ and satisfies $\varphi(c)= 1-\alpha/2+\alpha(c-1/2)$. There exists a unique area preserving homeomorphism $\widetilde g_1\in \widetilde{\mathrm{Diff}}^0_0(\T^2)$ such that $\widetilde g_1(x,y)= (x+\varphi(y), y)$, if $(x,y)\in\R\times[0,1]$. It sends the point $$\widetilde f_2(\alpha/2-\alpha(c-1/2),c)=\widetilde f_1(\alpha/2-\alpha(c-1/2),c)=(\alpha-\alpha (2c-1),c)$$ onto $(\alpha/2-\alpha(c-1/2)+1,c)$. Setting $\widetilde f_H=\widetilde g_1\circ \widetilde f_2$ and $\widetilde z_1= (\alpha/2-\alpha(c-1/2),c)$, one gets a map that satisfies all assertions formulated in Proposition \ref{ExampleConservative}. Note that the three assertions relative to the vector field are satisfied because the vector field $\widetilde z\mapsto \widetilde{f}_H(\widetilde{z})-\widetilde{z}$ is horizontal pointing to the right. \end{proof}

 Similarly we prove

 \begin{proposition} There exists $\widetilde{f}_V\in \widetilde{\mathrm{Diff}}^{0}_0(\T^2)$ such that:

\begin{enumerate}

\item $f$ is area preserving

\item  the fixed point set of $\widetilde{f}_V$ is $\widetilde V$;

\item if $\widetilde{z}\in \widetilde C\setminus\widetilde V$, the vector $\widetilde f_V(\widetilde z)-\widetilde z$ belongs to the cone of equation $\vert x\vert <y$;

\item if $\widetilde{z}\in \widetilde H\setminus\widetilde V$ and $\widetilde f(\widetilde{z})\not\in \widetilde H$, the vector $\widetilde f_V(\widetilde z)-\widetilde z$ belongs to the cone of equation $\vert x\vert <y$;

\item if $\widetilde{z}\in \widetilde H\setminus\widetilde V$ and $\widetilde f(\widetilde{z})\in \widetilde H$, the vector $\widetilde f_V(\widetilde z)-\widetilde z$ belongs to the half-plane of equation $x<y$;

\item  there exists $\widetilde{z}'_0\in\widetilde H$ such that $$\lim_{k\to -\infty}\widetilde{f}_V^k(\widetilde{z}'_0)=(1,0), \enskip \lim_{k\to +\infty}\widetilde{f}_V^k(\widetilde{z}'_0)=(0,0);$$

\item  there exists  $\widetilde z'_1\in \widetilde H$ such that $\widetilde{f}_V(\widetilde{z}'_1)=\widetilde{z}'_1+(0,1)$.

\end{enumerate}

\end{proposition}

It remains to prove that $\widetilde f=\widetilde f_V\circ \widetilde f_H$ satisfies the properties formulated in Theorem \ref{main4}.

Let us prove first that each vector $f(\widetilde z)-\widetilde z$, $\widetilde z\in\R^2$, is not a positive multiple of $(-1,-1)$.

\begin{itemize}

\item If $\widetilde{z}\in \widetilde C\setminus\widetilde H$ and $\widetilde f_H(\widetilde{z})\in \widetilde C\setminus\widetilde V$, then $\widetilde f_H(\widetilde z)-\widetilde z$ belongs to the cone of equation $\vert y\vert<x$ and $\widetilde f(\widetilde z)-\widetilde f_H(\widetilde z)$ to the cone of equation $\vert x\vert<y$, so $\widetilde f(\widetilde z)-\widetilde z$ belongs to the half-plane of equation $0<x+y$;

\item if $\widetilde{z}\in \widetilde C\setminus\widetilde H$ and $\widetilde f_H(\widetilde{z})\in \widetilde V$, then $\widetilde f(\widetilde z)-\widetilde z =\widetilde f_H(\widetilde z)-\widetilde z$ belongs to the cone of equation $\vert y\vert<x$;

\item if $\widetilde{z}\in \widetilde V\setminus\widetilde H$ and $\widetilde f_H(\widetilde{z})\in \widetilde C\setminus\widetilde V$, then $\widetilde f_H(\widetilde z)-\widetilde z$ belongs to the cone of equation $\vert y\vert<x$ and $\widetilde f(\widetilde z)-\widetilde f_H(\widetilde z)$ to the cone of equation $\vert x\vert<y$, so $\widetilde f(\widetilde z)-\widetilde z$ belongs to the half-plane of equation $0<x+y$;

\item if $\widetilde{z}\in \widetilde V\setminus\widetilde H$ and $\widetilde f_H(\widetilde{z})\in \widetilde V$, then $\widetilde f(\widetilde z)-\widetilde z =\widetilde f_H(\widetilde z)-\widetilde z$ belongs to the half-plane of equation $y<x$;

\item if $\widetilde{z}\in \widetilde H\setminus V$ and  $\widetilde f_V(\widetilde{z})\not\in \widetilde H$, then $\widetilde f(\widetilde z)-\widetilde z =\widetilde f_V(\widetilde z)-\widetilde z$ belongs to the cone of equation $\vert x\vert<y$;

\item if $\widetilde{z}\in \widetilde H\setminus V$ and  $\widetilde f_V(\widetilde{z})\in \widetilde H$, then $\widetilde f(\widetilde z)-\widetilde z =\widetilde f_V(\widetilde z)-\widetilde z$ belongs to the half-plane of equation $x<y$;

\item if $\widetilde{z}\in \widetilde H\cap \widetilde V$ then $\widetilde f(\widetilde z)=\widetilde z$.

\end{itemize}

If $v$ is a positive multiple of $(-1,-1)$, then $\widetilde f-v$ belongs to $ \widetilde{\mathrm{Diff}}^0_0(\T^2) $, is fixed point free and area preserving. As said in the previous section, $(0,0)$ is not an extremal point of the rotation set of $\widetilde f-v$. Let us prove that it does not belong to this set.

One gets a ``vertical topological line '' by adding to the broken line passing through $(0,0)$, $(1/2, \alpha/2)$ and $(0,1)$ all its translates by the vectors $(0,k)$, $k\in\Z$. This line is sent to its right by the maps $\widetilde f_1$, $\widetilde f_2$, $\widetilde f_H$. The fact that $\widetilde V$ is the fixed point set of $\widetilde f_V$ tells us that it is also sent to its  right by $\widetilde f$. Because the slopes are larger than $1$ in modulus, it is strictly sent on its right by the map $\widetilde f-v$. Similarly, one proves that 
the ``horizontal topological line '' obtained by adding to the broken line passing through $(0,0)$, $(1/2, \alpha/2)$ and $(1,0)$ all its translates by the vectors $(k,0)$, $k\in\Z$, is sent strictly above by  $\widetilde f-v$. Consequently, the rotation set of $\widetilde f-v$ is contained in $[0,+\infty)^2$ and cannot contain $(0,0)$.

By construction, one knows that $\widetilde f(\widetilde z_1)=\widetilde f_H(\widetilde z_1)= z_1+(1,0)$ and  $\widetilde f(\widetilde z'_1)=\widetilde f_V(\widetilde z'_1)= z'_0+(0,1)$. So the rotation set of $\widetilde f$ contains $(0,0)$, $(1,0)$ and $(0,1)$ and has non empty interior. Simlarly, one knows that $$\lim_{k\to -\infty}\widetilde{f}^k(\widetilde{z}_0)=(0,1), \enskip \lim_{k\to +\infty}\widetilde{f}^k(\widetilde{z}_0)=(0,0)$$ and $$\lim_{k\to -\infty}\widetilde{f}^k(\widetilde{z}'_0)=(1,0), \enskip \lim_{k\to +\infty}\widetilde{f}^k(\widetilde{z}'_0)=(0,0).$$

So, one can perturb $\widetilde f$  in a neighborhood of $\Z^2$ for the $C^0$ topology to get an area preserving map $\widetilde f'\in \widetilde{\mathrm{Diff}}^{0}_0(\T^2)$ such that 
$$\widetilde f'(\widetilde z_1)=z_1+(1,0), \enskip \widetilde f(\widetilde z'_1)=z'_1+(0,1)$$ and such that there exist positive integers  $q$ and $q'$ satisfying
$$\widetilde f'^{q}(\widetilde z_0)=z_0-(0,1), \enskip \widetilde f'^{q'}(\widetilde z'_1)=z'_1-(1,0).$$
The rotation set of $\widetilde f'$ contains the vectors $(0,1)$, $(1,0)$, $(-1/q, 0)$, $(0, -1/q')$, its interior contains $(0,0)$.

\vskip0.5truecm

{\it Acknowledgements:} We would like to thank Professor Robert Roussarie and Professor Freddy Dumortier for their help concearning the study of the local dynamics near isolated fixed points of analytic area-preserving diffeomorphisms of the plane.

\end{document}